\documentclass[12pt,a4paper]{amsart}

\usepackage{amssymb,enumerate}

\usepackage{graphicx,amssymb,amsfonts,epsfig,amsthm,a4,amsmath,url,verbatim}
\usepackage[latin1]{inputenc}
\newtheorem{thm}{Theorem}[section]
\newtheorem{cor}[thm]{Corollary}
\newtheorem{lem}[thm]{Lemma}

\newtheorem{prop}[thm]{Proposition}

\theoremstyle{definition}
\newtheorem{defn}[thm]{Definition}

\newtheorem{que}[thm]{Question}

\newtheorem{exe}[thm]{Example}

\newtheorem{rem}[thm]{Remark}

\numberwithin{equation}{section}

\newcommand{\N}{\mathbf{N}}

\newcommand{\Z}{\mathbf{Z}}
\newcommand{\R}{\mathbf{R}}

\newcommand{\Q}{\mathbf{Q}}

\newcommand{\eps}{\varepsilon}

\newcommand{\bfS}{\mathbf{S}}
\newcommand{\PC}{\mathrm{PC}}

\newcommand{\IET}{\mathrm{IET}}
\newcommand{\sing}{\mathrm{sing}}
\newcommand{\essup}{\mathrm{essupp}}

\newcommand{\mks}{\mathfrak{S}}

\newcommand{\mksh}{\mathfrak{S}^\star_0}
\newcommand{\mkst}{\mathfrak{S}^\star}

\newcommand{\bw}{{\,\bowtie}}

\begin{document}
\title[Realizations of piecewise continuous groups]{Realizations of groups of  piecewise continuous transformations of the circle}
\author{Yves Cornulier}%
\address{CNRS and Univ Lyon, Univ Claude Bernard Lyon 1, Institut Camille Jordan, 43 blvd. du 11 novembre 1918, F-69622 Villeurbanne}
\email{cornulier@math.univ-lyon1.fr}

\date{July 18, 2019}%\today}%February 19, 2019}%\today}%January 11, 2018}

%03Exx		Set theory
%	03E05  	Other combinatorial set theory
%	03E15  	Descriptive set theory
%  03Gxx		Algebraic logic
%	03G05  	Boolean algebras
%03Gxx		Algebraic logic
%	03G05  	Boolean algebras 
%05 combinatorics
%	05C63  	Infinite graphs
%	18B05 	Category of sets, characterizations
%06-XX			Order, lattices, ordered algebraic structures
%  06Exx		Boolean algebras (Boolean rings)
%	06E15  	Stone spaces (Boolean spaces) and related structures
%20Bxx		Permutation groups
%   20B07  	General theory for infinite groups
%	20B27  	Infinite automorphism groups
%	20B30   Symmetric groups
%	20F65  	Geometric group theory
%	20M18  	Inverse semigroups
% 	20B30  	Symmetric groups
%22Fxx		Noncompact transformation groups
%   22F05  	General theory of group and pseudogroup actions
%	22F50  	Groups as automorphisms of other structures
%	20M20  	Semigroups of transformations, etc.
%	20M30  	Representation of semigroups; actions of semigroups on sets	
%	37E05  	Maps of the interval (piecewise continuous, continuous, smooth)
%	37C85  	Dynamics of group actions other than ${\bf Z}$ and ${\bf R}$, and foliations

\begin{abstract}
We study the near action of the group $\PC$ of piecewise continuous self-transformations of the circle. Elements of this group are only defined modulo indeterminacy on a finite subset, which raises the question of realizability: a subgroup of $\PC$ is said to be realizable if it can be lifted to a group of permutations of the circle.

We show that every finitely generated abelian subgroup of $\PC$ is realizable.
We show that this is not true for arbitrary subgroups, by exhibiting a non-realizable finitely generated subgroup of the group of interval exchanges with flips. 

The group of (oriented) interval exchanges is obviously realizable (choosing the unique left-continuous representative). We show that it has only two realizations (up to conjugation by a finitely supported permutation): the left and right-continuous ones.
\end{abstract}

\maketitle

\section{Introduction}

\subsection{Context}
We deal with various groups of piecewise continuous transformations in dimension 1. The best known is the group of piecewise translations, better known as group of interval exchange transformations. 
Interval exchanges were introduced by Keane \cite{Kea}; they have mostly been studied in classical dynamics (iteration of a single transformation). Its study as a group notably starts in the determination by Arnoux-Fathi and Sah of its abelianization \cite{Ar}.
This study has been recently pursued, notably in work by C.\ Novak (e.g., \cite{Nov}), Dahmani-Fujiwara-Guirardel \cite{DFG,DFG2}, and Boshernitzan \cite{Bos}, see also \cite{Cpi}. Two outstanding problems about this group is whether it admits non-abelian free subgroups, a question attributed to A.\ Katok, and whether it is amenable \cite{Cbou}. Recent progress on this latter question is due to Juschenko-Monod \cite{JM}, subsequently improved by these two authors along with Matte Bon and de la Salle \cite{JMMS}. If we allow flips, we obtain a larger group, which is seldom studied, and usually not precisely defined. The questions of realizability, which we consider here, do not seem to have been considered, notably because defining interval exchanges with flips as a group is usually swept under the carpet.

\subsection{Set-up}

Let $\mathbf{S}$ be the circle $\R/\Z$.

\begin{defn}Let $\widehat{\PC^\bw}(\mathbf{S})$ be the group of permutations of $\mathbf{S}$ that are continuous outside a finite subset (it is indeed stable under inversion, by an easy argument).
\end{defn}

The group $\widehat{\PC^\bw}(\mathbf{S})$ includes the normal subgroup of finitely supported permutations $\mathfrak{S}_{\mathrm{fin}}(\mathbf{S})$ of $\mathbf{S}$. 

\begin{defn}We define $\PC^\bw(\mathbf{S})$ as the quotient group $\widehat{\PC^\bw}(\mathbf{S})/\mathfrak{S}_{\mathrm{fin}}(\mathbf{S})$.
\end{defn}

Thus, $\PC^\bw(\mathbf{S})$ is the group of all piecewise continuous permutations of $\mathbf{S}$, up to finite indeterminacy.

\begin{defn}
Let $\PC^+(\mathbf{S})$ be its subgroup of piecewise orientation-preserving transformations. Let $\PC^-(\mathbf{S})$ be the subset of $\PC^\bw(\mathbf{S})$ consisting of those piecewise-reversing transformations. Let $\PC^\pm(\mathbf{S})\subset\PC^\bw(\mathbf{S})$ be the disjoint union $\PC^+(\mathbf{S})\sqcup\PC^-(\mathbf{S})$.
\end{defn}

Thus $\PC^+(\mathbf{S})$ is a subgroup of index two in $\PC^\pm(\mathbf{S})$. Note that in contrast to what happens in self-homeomorphism groups, $\PC^+(\mathbf{S})$ has infinite index in $\PC^\bw(\mathbf{S})$, which we have to distinguish from $\PC^\pm(\mathbf{S})$.

\begin{defn}
Let $\IET^\bw(\mathbf{S})$ be the subgroup of $\PC^\bw(\mathbf{S})$ consisting of piecewise isometric elements (also called group of interval exchanges with flips). Define $\IET^+(\mathbf{S})=\PC^+(\mathbf{S})\cap \IET^\bw(\mathbf{S})$, the subgroup of piecewise translations, usually called group of interval exchanges. Also define $\IET^\pm(\mathbf{S})=\PC^\pm(\mathbf{S})\cap \IET^\bw(\mathbf{S})$.
\end{defn}

\begin{figure}[h]
\includegraphics[width=13cm]{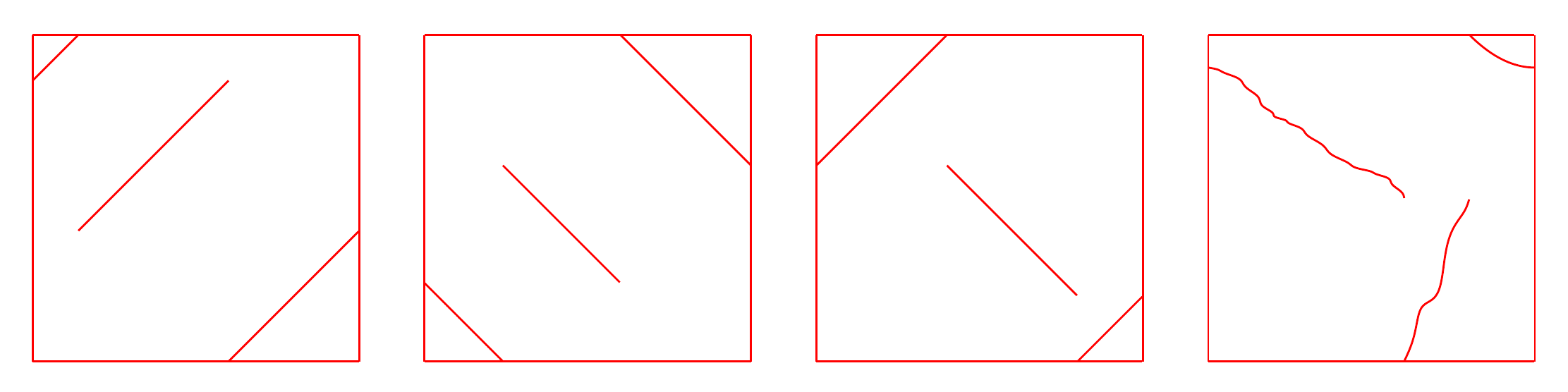}
\caption{Parameterizing the circle as an interval, examples of graphs of elements of $\PC^\bw(\mathbf{S})$. The first belongs to $\IET^+$; the second belongs to $\IET^-$, the third to $\IET^\bw\smallsetminus\IET^\pm$. The fourth is a more ``typical" element of $\PC^\bw(\mathbf{S})$. The value at breakpoints is not prescribed, as we consider their group elements as defined up to finite indeterminacy.}\label{fig1}
\end{figure}

We denote by $\pi$ the quotient group homomorphism $\widehat{\PC^\bw}(\mathbf{S})\to \PC^\bw(\mathbf{S})$.

\begin{defn}\label{reali_s1}
We say that a subgroup $\Gamma\subset\PC^\bw(\mathbf{S})$ is {\bf realizable} if it can be lifted to $\widehat{\PC^\bw}(\mathbf{S})$, i.e., if there exists a subgroup $\Lambda$ of $\widehat{\PC^\bw}(\mathbf{S})$ such that $\pi|_\Lambda$ is bijective.

More generally, a {\bf piecewise continuous near action} of a group $\Gamma$ on $\mathbf{S}$ means a homomorphism $\Gamma\to\PC^\bw(\mathbf{S})$. We call it {\bf realizable} if it can lifted to a homomorphism $\Gamma\to\widehat{\PC^\bw}(\mathbf{S})$.
\end{defn}

Note that every subgroup of a realizable subgroup is realizable.

\begin{exe}\label{i_finite_case}
Every finite subgroup of $\PC^+(\mathbf{S})$ is realizable. For instance, every element of order 2 has a lift of order 2. This easy fact is true in the broader context of near actions, see Remark \ref{finite_case}.
\end{exe}

\begin{exe}
The subgroup $\PC^+(\mathbf{S})$ is realizable. Indeed, we can lift $f$ to its unique left-continuous representative, and in restriction to $\PC^+(\mathbf{S})$, taking this lift is a group homomorphism. While taking the unique left-continuous representative makes sense for all $f\in\PC^+(\mathbf{S})$, this unique lift is bijective if and only if $f\in\PC^\pm(\mathbf{S})$. However this does not make $\PC^\pm(\mathbf{S})$ realizable: it is easy to find in $\IET^-$ a transformation of order 2 that has no lift of order 2 that is either left-continuous or right-continuous (see Lemma \ref{nonliftnew}).
\end{exe}

That $\PC^+(\mathbf{S})$ is realizable makes it (and its subgroups) easier to define, since one can refer to piecewise continuous, left-continuous permutations of the circle. This artifact makes the definition shorter (since one does not have to mod out finitely supported permutations), and often explains the restriction to the piecewise orientation-preserving case, in many settings where this is not really used.

\subsection{Non-realizability and restriction results}

The first main result of this paper is a non-realizability theorem. 

\begin{thm}[Theorem \ref{nonrepm}]
The group $\PC^\bw(\mathbf{S})$ is not realizable. More precisely, $\IET^\pm$ is not realizable, and even has a finitely generated subgroup that is not realizable.
\end{thm}

Actually, the result also holds, with some technical cost, with the stronger conclusion ``not stably realizable" (Theorem \ref{t_stabr}). The latter is a more natural notion, see \S\ref{preli}; rather than defining it in this introduction, let us pinpoint that it is equivalent to the assertion that, for every nonempty open interval, the group of interval exchanges with flips that induce the identity on the complement of $I$, is not realizable. The latter can be viewed as the group of interval exchanges with flips of $I$. So the failure of stable realizability means that even making use of those additional points in the complement, does not allow to realize the action.

Our approach also provides, with further work, a result in the piecewise orient\-ation-preserving case. For a subgroup $\Gamma$ of $\PC^+(\mathbf{S})$, we denote by $\Gamma_{\mathrm{left}}$ (respectively $\Gamma_{\mathrm{right}}$) the group of left-continuous (resp.\ right-continuous) representatives of elements of $\Gamma$.

Recall that the {\bf $\Q$-rank} of an abelian group $A$ is the dimension of $A\otimes_\Z\Q$; for a finitely generated abelian group $A\simeq\Z^d\times (\mathrm{finite})$, this is just~$d$.

\begin{thm}[Theorem \ref{t2lifts} and Corollary \ref{2lifts_cor}]
Modulo conjugation by finitely supported permutations, the only subgroups lifting $\Gamma=\IET^+$ are $\Gamma_{\mathrm{left}}$ and $\Gamma_{\mathrm{right}}$. The same conclusion holds when $\Gamma$ is:
\begin{itemize}
\item any subgroup of $\PC^+(\mathbf{S})$ that includes $\IET^+$ (e.g., the piecewise affine orientation-preserving subgroup);
\item for any subgroup $\Lambda$ of rotations of $\Q$-rank $\ge 2$, the subgroup $\IET^+_\Lambda$ of interval exchanges with singularities and translation lengths in $\Lambda$.
\end{itemize}
\end{thm}

\subsection{Realizability results}

Realizability makes sense in a much more general setting (near actions on sets), see \S\ref{preli}. As already mentioned, finite groups are always realizable, see Remark \ref{finite_case}. This is also true for free subgroups of the quotient by finitely supported permutations, obviously. On the other hand, this is not true for general near actions of $\Z^2$, as a variety of examples in \cite{Cn} show. More precisely, it is easy to find two permutations of a set that commute as near permutations (i.e., their commutator is finitely supported), but they cannot be perturbed (i.e., multiplied by finitely supported permutations) so that the resulting permutations commute. Nevertheless, we show here that such phenomena cannot arise in the context of piecewise continuous near actions on the circle.

\begin{thm}\label{fga}
Any finitely generated abelian subgroup of $\PC^\bw(\mathbf{S})$ is realizable.
\end{thm}

The proof of Theorem \ref{fga} is not direct. It makes use the fact that the near action on the circle can be viewed as a projection of a genuine action, namely obtained by doubling all points (the Denjoy blow-up).

\begin{rem}The near action of $\PC^\bw(\mathbf{S})$ is completable, in the sense that there exists an action on a set $X$ (some huge non-Hausdorff connected compact 1-dimensional manifold), in which $\mathbf{S}$ sits as a commensurated subset, so that the induced near action on $\mathbf{S}$ is the given one. This observation comes from \cite{Cpi}.
\end{rem}

\subsection{Outline}
In \S\ref{preli}, we recall the language of near actions introduced in \cite{Cn}, and prove some preliminary results for which the 1-dimensional context does not play any role. In \S\ref{rfga}, we conclude the proof of the positive results, on piecewise continuous near actions of finitely generated abelian groups. In \S\ref{nriet}, we prove the non-realizability results. 

\subsection{Open questions}

\begin{que}
Let $\Lambda$ be a (virtually) infinite cyclic subgroup of $\bfS$. Is the near action of $\IET^\bw_\Lambda$ on $\bfS$ realizable? stably realizable?
\end{que}

\begin{que}
Is there a subgroup of $\PC^\bw(\mathbf{S})$ that is stably realizable but not realizable?
\end{que}

\begin{que}
Is every solvable subgroup of $\PC^\bw(\mathbf{S})$ realizable? stably realizable. Same question for a finitely generated subgroup, and/or assuming it is included in $\IET^\bw(\mathbf{S})$.
\end{que}

\begin{que}
Does the near action of $G=\PC^\bw(\mathbf{S})$ (or $\IET^\bw$) on $\mathbf{S}$ have a zero Kapoudjian class? This vanishing would mean that $G$ can be lifted to the quotient of $\widehat{\PC^\bw}(\mathbf{S})$ by its subgroup of {\it alternating} finitely supported permutations. (See \cite[\S 8.C]{Cn} for more on the Kapoudjian class.) 
\end{que}

\begin{que}
Is $\IET^\bw$ realizable in the measurable setting (in the sense of Mackey \cite{Ma})?
Namely, denote by $\hat{G}$ the group of permutations $f$ of $\bfS$ such that both $f$ and $f^{-1}$ are Lebesgue-measurable and such that $f$ is measure-preserving. Let $G$ be its quotient by the subgroup of those $f$ that are identity on a subset of measure $1$. Is the surjective homomorphism $\hat{G}\to G$ split in restriction to $\IET^\bw\subset G$?
\end{que}

\noindent {\bf Acknowledgement.} I am very grateful to the referee for a careful reading and pointing out serious mistakes in a first version.

\tableofcontents

\section{Preliminaries on general near actions}\label{preli}

{\bf Convention} (contain/include): on the one hand, the assertion ``$a\in X$" is written: ``$a$ {\bf belongs} to $X$", or ``$X$ {\bf contains} $a$"; on the other hand ``$X\subset Y$" is written ``$X$ is {\bf included} in $Y$", or ``$Y$ {\bf includes} $X$".

\subsection{Basic definitions and facts} By ``cofinite subset" we mean ``subset with finite complement".
Essentially following Wagoner \cite[\S 7]{Wa} (see also \cite{Cn} for a detailed historical account), a cofinite-partial bijection of a set $X$ is a bijection between two cofinite subsets of $X$. If we identify any two such cofinite-partial bijections when they coincide on a cofinite subset, we obtain the {\bf near symmetric group} of $X$, denoted by $\mkst(X)$. Its elements, namely cofinite-partial bijections modulo cofinite coincidence, are called {\bf near permutations} of $X$. There is a canonical homomorphism from the group $\mks(X)$ of permutations of $X$ to $\mkst(X)$. Its kernel is the subgroup $\mathfrak{S}_{\mathrm{fin}}(X)$ of finitely supported permutations. Its image $\mksh(X)$ consists by definition of {\bf balanced near permutations} and is called the {\bf balanced symmetric group} of $X$. Given a cofinite-partial bijection $f:X\smallsetminus F_1\simeq X\smallsetminus F_2$, the number $|F_2|-|F_1|$ is called the {\bf index} $\phi_X(f)$ of $f$. The index map factors through a group homomorphism $\phi_X:\mkst(X)\to\Z$, called {\bf index homomorphism}, whose kernel is precisely $\mksh(X)$. If $X$ is infinite, the index homomorphism $\phi_X$ is surjective, so that the cokernel $\mkst(X)/\mksh(X)$ is infinite cyclic.

\begin{defn}[\cite{Cn}]
A {\bf near action} of a group $G$ on a set $X$ is the datum of a homomorphism $\alpha:G\to\mkst(X)$; then $X$ is called a {\bf near $G$-set}. The near action is said to be {\bf balanced} if $\alpha$ is valued in $\mksh(X)$, or equivalently if the {\bf index homomorphism} $\phi_X\circ\alpha\in\mathrm{Hom}(G,\Z)$ of the near action $\alpha$ is zero. 

A near $G$-set $X$ as above (or the near action of $G$ on $X$) is said to be $\star$-faithful if $\alpha$ is injective. It is said to be near free if for every $g\in G\smallsetminus\{1\}$, the set of points fixed by $g$ (which is well-defined modulo symmetric difference with finite subsets) is finite.
\end{defn}

For subsets $U,V$ of a set $X$, we write $U\stackrel{\star}= V$ and say that $U$ and $V$ are near equal if $U\bigtriangleup V$ is finite. We write $U\subset^\star V$ and say that $U$ is near included in $V$ if $U\smallsetminus V$ is finite: thus $U\stackrel{\star}= V$ if and only if both $U\subset^\star V$ and $V\subset^\star U$.
 For maps $f,g:X\to Y$, we write $f\stackrel{\star}= g$ if $f$ and $g$ coincide outside a finite subset: this means that among subsets of $X\times Y$, the graphs of $f$ and $g$ are near equal.

While the notion of invariant subset for a group action does not pass to near actions, we have a notion of commensurated subset. Namely, given a near action of $G$ on $X$, a subset $Y\subset X$ is $G$-{\bf commensurated} if for every $g\in G$, we have $gY\stackrel{\star}=Y$. (Note that $gY$ is well defined modulo near equality.)
 Thus, $G$-commensurated subsets are the same as near $G$-subactions. 
 
A near $G$-set $X$ is said to be {\bf 1-ended} if $X$ is infinite and its only commensurated subsets are finite or have finite complement. That is, $X$ is not decomposable as disjoint union of two infinite near $G$-subactions. A near $G$-set is said to be {\bf finitely-ended} if it had only finitely many commensurated subset modulo near equality. This means that it decomposes as a finite disjoint union of 1-ended commensurated subsets (which are well-defined modulo near equality), their number is called the number of ends of the near $G$-set $X$.

Recall that a map between sets is said to be {\bf proper} if it has finite fibers. A map $f:X\to Y$ (with cofinite domain of definition) between near $G$-sets (either assuming $f$ to be proper, or $Y$ to be a given $G$-set) is said to be {\bf near $G$-equivariant} if for every $g\in G$, the set of $x\in X$ such that $f(gx)=gf(x)$ is cofinite in $X$. Note that we choose here, for given $g$, self-maps $x\mapsto gx$ and $y\mapsto gy$ of $X$ and $Y$; the condition does not depend on these choices when $f$ is proper, or when the $G$-action on $Y$ is fixed (with neither assumption the notion of near equivariant map is ill-defined).
The map $f$ is said to be a {\bf near isomorphism} if there exists another such near $G$-equivariant map $f':Y\to X$ such that $f\circ f'\stackrel{\star}=\mathrm{Id}_Y$ and $f'\circ f\stackrel{\star}=\mathrm{Id}_X$; in this case $X$ and $Y$ are said to be {\bf near isomorphic} near $G$-sets.

\begin{defn}\label{d_stab_real}
A near action of a group $G$ on a set $X$ is:
\begin{itemize}
\item {\bf realizable} if the homomorphism $G\to\mkst(X)$ comes from a homomorphism $G\to\mks(X)$;
\item {\bf {[finitely]} stably realizable} if there exists a [finite] set $Y$ with trivial near action such that the near action on $X\sqcup Y$ is realizable;
\item {\bf completable} if there exists a near $G$-set $Y$ such that the near action on $X\sqcup Y$ is realizable.
\end{itemize}
\end{defn}

Obviously, realizable implies finitely stably realizable, which implies stably realizable, which implies completable. Examples in \cite{Cn} show that none of the converse implications holds, and that there exist non-completable near actions. All these examples are taken with $G=\Z^2$, except for the difference between stably realizable and finitely stably realizable: indeed it is established in \cite{Cn} that these two notions are equivalent for finitely generated groups.

\begin{rem}\label{finite_case}
Every near action of a finite group is realizable \cite[Proposition 5.A.1]{Cn}. For completeness let us sketch the argument. It is enough to check that every finite subgroup $F\subset\mkst(X)$ can be lifted. Since $\mathrm{Hom}(F,\mathbf{Z})=\{0\}$, we have $F\subset \mathfrak{S}(X)/\mathfrak{S}_{\mathrm{fin}}(X)$. Hence we can find a finitely generated subgroup $F'\subset\mathfrak{S}(X)$ projecting onto $F$, with kernel $N$ of finite index. Since $N$ is finitely generated and acts by finitely supported permutations, the union $Y$ of supports of its elements is finite; it is also $F'$-invariant. Changing the $F'$-action to be trivial on $Y$, we obtain another subgroup, which is a lift of $F$.
\end{rem}

Let $G$ be a finitely generated group and $X$ a near $G$-set. Fix a finite generating subset $S$ of $G$, and lift each $s\in S$ to cofinite-partial bijection $x\mapsto sx$ of $X$. The corresponding {\bf near Schreier graph} consists in joining $x$ to $sx$ for all $s\in S$. The near action is said to be {\bf of finite type} if the near Schreier graph has finitely many components. Routine arguments in \cite{Cn} show that this does not depend on the choices.

\subsection{Rigidity of 1-ended actions}

The following rigidity results are used both in \S\ref{rfga} and \S\ref{nriet}. They are borrowed from \cite[\S 7]{Cn}; we include most proofs for completeness and for convenience.

We first recall, given an infinite finitely generated abelian group $A$, that the 1-ended commensurated subsets of $A$ are, up to finite symmetric difference
\begin{itemize}
\item $A$ itself if $A$ has $\Q$-rank $\ge 2$;
\item $f^{-1}(\N)$ and $f^{-1}(-\N)$ if $A$ has $\Q$-rank 1, where $f$ is a homomorphism of $A$ onto $\Z$.
\end{itemize}

\begin{lem}\label{neareqend}
Let $G$ be a group and $Y$ a 1-ended commensurated subset (for the left action of $G$ on itself). 
Let $f:Y\to G$ be a near equivariant map, in the sense that for every $g$, for all but finitely many $h$ we have $f(gh)=gf(h)$. Then there exists a unique $s\in G$ such that for all but finitely many $g\in Y$ we have $f(g)=gs$.
\end{lem}
\begin{proof}
Uniqueness holds because $Y$ is infinite. Let us prove the existence. Define $u(g)=g^{-1}f(g)$. By assumption, for every $g\in G$, for all but finitely many $h\in Y$ we have $u(gh)=u(h)$. Since $Y$ is 1-ended, this implies that $u$ is constant outside a finite subset of $Y$, which exactly means the conclusion.
\end{proof}

Given two actions $\alpha,\alpha'$ of a group on a set, we say that they are finite perturbations of each other if they induce the same near action. In other words, this means that for every $g$, the permutations $\alpha(g)$ and $\alpha'(g)$ coincide on a cofinite subset (depending on $g$).

\begin{prop}\label{unchangedconj}
Let $A$ be a finitely generated abelian group. Let $X$ be an $A$-set, and let $X_{\ge 2}$ be the union of $A$-orbits of at least quadratic growth. Then any finite perturbation of the action is conjugate, by a finitely supported permutation, to an action that is unchanged on $X_{\ge 2}$.
\end{prop}
\begin{proof}
We call ``orbits" those orbits of the original action $\alpha$ and ``new orbits" those of the perturbed action $\alpha'$.

First assuming that the number $n$ of $G$-orbits of at least quadratic growth is finite, we argue by induction on $n$. The case $n=0$ is trivial; assume $n\ge 1$. Let $Y$ be an orbit of at least quadratic growth. Let $B\subset A$ be point stabilizer of $Y$, that is, the stabilizer of some/any element of $Y$. Then $Y$ is a 1-ended commensurated subset of at least quadratic growth (for either action), and hence has finite symmetric difference with a new orbit $Y'$, also of at least quadratic growth, with point stabilizer $B'$; choose $y\in Y\cap Y'$. Since every element of $B$ acts with finite support on $Y'$ (for the original, and hence for the new action), and since the new action of $A/B'$ is free on $Y'$ and $Y'$ is infinite, we have $B\subset B'$. Arguing the other way round, we have $B'\subset B$, and hence $B=B'$. Extend the identity map of $Y\cap Y'$ to a map $f:Y\to Y'$. Then $f$ is near equivariant from the original to the new action. The map $g\mapsto \alpha(g)y$ induces an equivariant bijection $u:A/B\to Y$, and similarly the map $g\mapsto \alpha'(g)y$ induces an equivariant (for the new action) bijection $u:A/B\to Y'$. Define $f'={u'}^{-1}\circ f\circ u$.
Then $f'$ is a near equivariant (proper) self-map of $A/B$. By Lemma \ref{neareqend}, $f'$ coincides outside a finite subset with a translation $g\mapsto gh_0$. In particular, $f'$ has index zero, and hence $f$ has index zero. This means that $Y\smallsetminus Y'$ and $Y'\smallsetminus Y$ have the same cardinal. Hence, after conjugating the new action by a finitely supported permutation, we can suppose that $Y=Y'$.

Now $Y=Y'$. Since for all but finitely many $g\in A/B$ we have $f'(g)=gh_0$, setting we have, for all but finitely $g\in A/B$, the equality $u(g)=u'(gh_0)$. Setting $y'=\alpha'(h_0)y$, this rewrites as: for all but finitely many $g\in A/B$, we have $\alpha(g)y=\alpha'(g)y'$. We define a self-map $w$ of $Y$ by $w(\alpha(g)y)=\alpha'(g)y'$. So, $w$ is injective and equals the identity outside a finite subset, hence it is surjective. Thus $w$ is a finitely supported permutation. Then for all $g\in G$ and $x\in X$, choosing $h\in G$ such that $x=\alpha(h)y$, we have
\[w(\alpha(g)x)=w(\alpha(gh)y)=\alpha'(gh)y'=\alpha'(g)w(x),\]
thus $w^{-1}\circ\alpha'(g)\circ w=\alpha(g)$. Hence conjugating by $w$ yields the original action on $Y$. Now removing $Y$ we can continue by induction.

Finally, when $X$ has infinitely many orbits, the original and the new action differ only on finitely many orbits, so the result follows from the case of an action with finitely many orbits.
\end{proof}

The following is established in \cite[Theorem 7.C.1]{Cn} and will be used in \S\ref{nriet}. The case when $G$ is a finitely generated abelian group of $\Q$-rank $\ge 2$ is covered by Proposition \ref{unchangedconj} and is enough for most of the purposes, so we do not prove the general case, which relies on the same ideas. Recall that a group $G$ (not necessarily finitely generated) is said to be {\bf 1-ended} if it is a 1-ended $G$-set (under the left action).

\begin{prop}\label{per1e}
Let $G$ be a 1-ended group that is not locally finite, and $X$ a $G$-set. Then

\begin{enumerate}
\item\label{pe1} Suppose that $G$ acts freely on $X$. Then any finite perturbation of the action is conjugate, by a unique finitely supported permutation, to the original action.

\item\label{pe2} Suppose instead that $G$ acts freely on $Y=X\smallsetminus X^G$ (where $X^G$ is the set of points fixed by all of $G$). 
Then any finite perturbation of the action is conjugate, by a finitely supported permutation, to an action that is unchanged on $Y$.\qed
\end{enumerate}
\end{prop}

It is an old result of W. Scott and L. Sonneborn \cite[Theorems 1 and 2]{ScS} that an abelian group $A$ is 1-ended if and only if it satisfies one the following conditions: $A$ is uncountable, or $A$ has $\Q$-rank $\ge 2$, or $A$ is infinitely generated of $\Q$-rank $1$. (The remaining abelian groups are countable locally finite or virtually cyclic.)

\subsection{Facts on realizability for finitely generated abelian groups}

The next two propositions are borrowed from \cite{Cn}. They are used to prove Corollary \ref{yyrea}, which is used in \S\ref{rfga}. Again, we include the proofs for completeness. 

\begin{prop}\label{complab}
Let $A$ be a finitely generated abelian group. Let $X$ be a near $A$-set. Equivalences:
\begin{itemize}
\item $X$ is finitely stably realizable;
\item $X$ is completable, and the index homomorphism of the near $A$-subset $X^B$ of $B$-fixed points is zero for every subgroup $B$ of $A$.
\end{itemize}
\end{prop}
Note that the $A$-commensurated subset $X^B$ is defined up to $\stackrel{\star}=$, and hence whether it is balanced is a well-defined notion.

\begin{proof}This is proved in \cite{Cn}; for completeness we include the proof. The forward implication is clear; suppose that the condition holds.
Complete $X$ to an $A$-set $X\sqcup Y$, so that $X$ meets every orbit. Since $X$ has finite boundary in the Schreier graph, all but finitely many orbits are included in $X$. We can then restrict to those orbits not including $X$, and thus assume that $X$ has finite type.
 
So, suppose that $X$ has finite type and satisfies the second condition. Since every transitive $A$-set is finitely-ended, and $X$ is completable, $X$ is finitely-ended.

For each subgroup $B$ of $A$ such that $A/B$ has $\Q$-rank 1, the number of homomorphisms of $A/B$ onto $\Z$ is 2, and we choose one of the two as $u_B$. Then the classification of 1-ended completable near $A$-sets up to near isomorphism is as follows: it consists of those $E_B=A/B$ when $B$ ranges over subgroups of $A$ such that $A/B$ has $\Q$-rank $\ge 2$, those $E_B^+=u_B^{-1}(\N)$ and $E_B^-=u_B^{-1}(-\N)$ when $B$ ranges over subgroups of $A$ such that $A/B$ has $\Q$-rank $1$. 

Let $B$ be maximal among subgroups such that $X^B$ is infinite. Then $X^B$ satisfies the balanced assumption: indeed, by maximality of $B$, for every subgroup $B'$ of $A$, we have $(X^B)^{B'}\stackrel{\star}=X^B$ if $B'\subseteq B$, and $(X^B)^{B'}\stackrel{\star}=\emptyset$ otherwise.
Since both $X^B$ and $X$ satisfy the balanced assumption, so does $X\smallsetminus X^B$. Hence, if $X\smallsetminus X^B$ is infinite, we can argue by induction on the number of ends to deduce that $X$ is finitely stably realizable.

Now we assume that $X^B$ is cofinite in $X$. So, using the maximality of $B$, the near $A$-set $X$ is a completable near free near $A/B$-set of finite type. Hence, it is near isomorphic to a commensurated subset of some free $A/B$-action with finitely many orbits. If $A/B$ has $\Q$-rank $\ge 2$, this implies that $X$ is near isomorphic to some disjoint union of copies of $E_B$, and hence is stably realizable. If $A/B$ has $\Q$-rank $1$, this implies that $X$ is near isomorphic to the disjoint union of $k$ copies of $E_B^+$ and $\ell$ copies of $E_B^-$. Note that the index homomorphism is additive under disjoint unions, and is opposite and nonzero for $E_B^+$ and $E_B^-$. That the index homomorphism of $X$ vanishes then implies that $k=\ell$. So $X$ is near isomorphic to the disjoint union of $k$ copies of $A/B$, so is finitely stably realizable.
\end{proof}

\begin{prop}\label{starepasre}
Let $A$ be a finitely generated abelian group. A near $A$-set $X$ is finitely stably realizable but not realizable if and only if for some nonempty finite set $F$, the disjoint union $X\sqcup F$ is realizable as an action with only orbits of at least quadratic growth.
\end{prop}
\begin{proof}
$\Rightarrow$ Suppose that $X\sqcup F$ is realizable for some nonempty finite set $F$, with $F$ of minimal cardinal, and assume $X$ not realizable, so $F$ is not empty. Fix a realization. Then it has no finite orbit, because this would contradict the minimality of $F$. Also it has no orbit of linear growth: indeed, an orbit of linear growth, say isomorphic to a quotient of $A$ isomorphic to $\Z\times K$ with $K$ finite, is balanceably isomorphic to itself minus $|K|$ points. Again this contradicts the minimality of $F$, and hence the realization has only orbits of at least quadratic growth.

Conversely, for an $A$-set $X$ with only orbits of at least quadratic growth, let us show that $X$ minus any nonempty finite set is not realizable. Equivalently, let us show that every realization of the near action on $X$ has no finite orbit. This follows from Proposition \ref{unchangedconj} (with $X=X_{\ge 2}$).
\end{proof}

\begin{cor}\label{yyrea}
Let $A$ be a finitely generated abelian group. Let $X$ be a near $A$-set. Suppose that for some $n\ge 1$, the disjoint union $nX$ of $n$ copies of $X$ is realizable (resp.\ finitely stably realizable). Then so is $X$.
\end{cor}
\begin{proof}
The assumption immediately implies that $X$ is completable.

Then the result for finite stable realizability immediately follows from the criterion of Proposition \ref{complab}. Now assume that $nX$ is realizable. Hence $X$ is finitely stably realizable. Assuming by contradiction that $X$ is not realizable, there exists, by the forward implication in Lemma \ref{starepasre}, a nonempty finite set $F$ such that $X\sqcup F$ is realizable as an action with only orbits of quadratic growth. Hence $nX\sqcup nF$ has the same property. By the reverse implication in Lemma \ref{starepasre}, we deduce that $nX$ is not realizable, a contradiction. 
\end{proof}

\begin{rem}
In contrast, in \cite{Cn}, examples of groups $G$ are given for which there exists a near $G$-set $X$ that is not stably realizable, such that $X\sqcup X$ is realizable. Such groups can both be chosen to be finitely generated (some well-chosen amalgam of two finite groups), or abelian (the quasi-cyclic group $\Z[1/2]/\Z$). 
\end{rem}

\subsection{A realizability criterion}\label{s_recri}

Applying results of the previous subsections, we obtain the following results (which are not in \cite{Cn}), and are used in \S\ref{rfga} to obtain the realizability result.

\begin{lem}\label{covdecompo}
Let $A$ be a finitely generated abelian group and $X$ an $A$-set with finitely many orbits. Let $\tau$ be a fixed-point-free permutation of $X$, of order 2. Suppose that $\tau$ is near $A$-equivariant, and suppose that ($\ast$) for every $g\in A$, the set of $x\in X$ such that $gx=\tau x$ is finite. Then there exist $n\ge 1$ and subsets $X_0,\dots,X_n$ of $X$ such that we have the partition
\[X=X_0\sqcup \tau(X_0)\sqcup X_1\sqcup \tau(X_1)\sqcup\cdots \sqcup X_n\sqcup \tau(X_n),\]
such that $X_0$ is finite, and such that each $X_i$ for $i\ge 1$ is $A$-commensurated and 1-ended.
\end{lem}
\begin{proof}
Decompose $X$ as finite union of 1-ended subsets. By near equivariance of $\tau$, the permutation $\tau$ permutes these subsets up to finite symmetric difference. Removing finite subsets, we can find $m\ge 0$ and a partition $X=Z_0\sqcup\dots \sqcup Z_m$ such that $Z_0$ is finite, each $Z_i$ is 1-ended for $m\ge 1$, and $\tau$ permutes the $Z_i$.
We claim that $\tau(Z_i)\neq Z_i$ for each $i\ge 1$. 

First assume that $Z_i$ has at least quadratic growth. Then there exists a 1-ended $A$-orbit $Ax_0$ having finite symmetric difference with $Z_i$; let $B$ be the stabilizer of $x_0$. Define $\tau':Ax_0\to Ax_0$ coinciding with $\tau$ on $\tau^{-1}(Ax_0)$. Then $\tau'$ can be viewed as a near equivariant self-map of $A/B$. Hence, by Lemma \ref{neareqend}, it coincides outside a finite subset with a translation. This translation cannot be trivial because $\tau$ is fixed-point-free. Hence, for some $s\in A\smallsetminus B$, we have $\tau(x)=sx\neq x$ for all but finitely many $x\in Ax_0$. This contradicts ($\ast$).

The other case is when $Z_i$ has linear growth. Then there exists $x_0\in X$, with stabilizer $B$, such that $A/B$ has $\Q$-rank 1, and a homomorphism $f$ from $A/B$ onto $\Z$, such that, denoting $A^+=f^{-1}(\N)$, $Z_i$ has finite symmetric difference with $A^+x_0$. Again by Lemma \ref{neareqend}, we see that $\tau$ coincides outside a finite subset of $Z_i$ with a translation and get a contradiction.

We conclude by defining $X_0\subset Z_0$ so that $Z_0=X_0\sqcup \tau(X_0)$, defining $I$ as a maximal subset of $\{1,\dots,m\}$ such that $\bigcup_{i\in I}Z_i$ is disjoint to its image by $\tau$, and enumerate the $Z_i$ for $i\in I$ as $X_1,\dots,X_n$. Since each $Z_i$ for $i\ge 1$ is disjoint to its image by $\tau$, the $X_i$ and their images by $\tau$ indeed cover $X$.
\end{proof}

\begin{thm}\label{rearran}
Let $A$ be a finitely generated abelian group. Let $X$ be an $A$-set, and $Y$ a near $A$-set. Let $\tau$ be a fixed-point-free near $A$-equivariant self-map of $X$, with $\tau^2=\mathrm{id}_X$. Suppose ($\ast$) that for every $g\in A$, the set of $x\in X$ such that $gx=\tau(x)$ is finite. Let $f:X\to Y$ be a surjective near equivariant map such that for every $x\in X$ we have $f^{-1}(\{f(x)\})=\{x,\tau(x)\}$. Then
\begin{enumerate}
\item\label{reacov} $Y$ is a realizable near $A$-set;
\item\label{balyy} if $X$ has finitely many $A$-orbits, then $X$ is balanceably isomorphic, as near $A$-set, to $Y\sqcup Y$.
\end{enumerate}
\end{thm}
\begin{proof}
The assertion (\ref{balyy}) is immediate from Lemma \ref{covdecompo}: indeed, in restriction to $U=X_0\sqcup\dots \sqcup X_n$, $f$ is near equivariant and bijective, and also in restriction to $\tau(U)$.

Now let us deduce (\ref{reacov}). Let $X'$ be the union of all orbits $Z$ such that $f$ and $\tau$ equivariant on both $Z$ and $\tau(Z)$; so $X'$ is $\tau$-invariant. So $f(X')$ is isomorphic as near $A$-set to the quotient of an $A$-set, and hence is realizable.

If $X''$ is the complement of $X'$, then $Y''=f(X'')$ is the complement of $f(X')$, and $X''$ consists of finitely many $A$-orbits. Hence we have to prove that $Y''$ is realizable. By (\ref{balyy}), $Y''\sqcup Y''$ is realizable. Finally by Corollary \ref{yyrea}, $Y''$ is realizable, and hence so is $Y=Y''\sqcup f(X')$. 
\end{proof}

\begin{rem}
Here is a counterexample to the statement of Theorem \ref{rearran} with ($\ast$) removed. Let $K$ be the Klein group of order 4, and $u,v$ two distinct elements of order 2 in $K$, and $F$ the subgroup generated by $u$. Let $A$ be the group $F\times\Z$, and $X=K\times\Z$, which is thus a free $A$-set with two orbits. Define a permutation of order 2 of $X$ by $\sigma(g,n)=(ug,n)$ for $n<0$ and $\sigma(g,n)=(vg,n)$ for $n\ge 0$. So $\sigma$ commutes with the action of $u$ and near commutes with the action of $\Z$, in the sense that the commutator of $\sigma$ with the generator $(g,n)\mapsto (g,n+1)$ has finite support. Hence, the quotient by the $u$-action is naturally a near $A$-set. It can be identified to $F\times\Z$, where $\Z$ acts by shifting, while $u$ acts by $(g,n)\mapsto (g,n)$ for $n<0$ and $\mapsto (ug,n)$ for $n\ge 0$. This near action is not stably realizable (this is the very first example in \cite{Cn}). 
\end{rem}

\section{Realizability of piecewise continuous near actions of finitely generated abelian groups}\label{rfga}

We now use the results of \S\ref{s_recri} (namely Theorem \ref{rearran}) to prove Theorem \ref{fga}.

\subsection{The ``true" definition of $\PC^\bw(\bfS)$}

Let $X$ be a Hausdorff topological space. The group $\PC(X)$ of {\bf near self-homeomorphisms} of $X$ consists of those elements of $\mkst(X)$ that have a representative that is a homeomorphism between two cofinite subsets.

Let $\widehat{\PC_0}(X)$ be the subgroup of permutations $f$ of $X$ such that both $f$ and $f^{-1}$ are continuous outside a finite subset. There is a canonical homomorphism $\widehat{\PC_0}(X)\to\PC(X)$; its image $\PC_0(X)$ equals $\PC(X)\cap\mksh(X)$, and its kernel consists of finitely supported permutations of $X$.

A basis remark is that $\PC_0(X)$ is a proper subgroup of $\PC(X)$ if and only if there exist two finite subsets $F,F'$ of $X$ with $|F|<|F'|$ such that $X\smallsetminus F$ and $X\smallsetminus F'$ are homeomorphic. 

For instance, this holds when $X$ is infinite discrete, or when $X$ is a Cantor space. Nevertheless, we have $\PC_0(X)=\PC(X)$ for $X=\bfS$: indeed, $\bfS$ minus $n$ points is homeomorphic to $\bfS$ for $n=0$ and to the disjoint union of $n$ copies of $\R$ when $n\ge 1$, so its topological type retains $n$. Hence, $\PC(\bfS)=\PC_0(\bfS)=\PC^\bw(\bfS)$.

\begin{rem}
In \cite{Cpi}, it is proved that the near action of $\PC(X)$ on $X$ is completable as soon as $X$ has no isolated point. This notably applies to $X=\bfS$.
\end{rem}

\subsection{The Denjoy blow-up}

Let $\bfS$ denote the circle $\R/\Z$. Let $\bfS^\pm$ denote the ``Denjoy blow-up" of $\bfS$ at all points. As a set, it can simply be defined as the Cartesian product $\bfS\times\{\pm 1\}$, where we write $x^+$ and $x^-$ for the elements $(x,1)$ and $(x,-1)$. For $y=(x,\eps)\in\bfS^\pm$, we write $\hat{y}=(x,-\eps)$ and $\bar{y}=x$.

It turns out that there are two natural compact Hausdorff topologies on this blow-up. The first is the product topology. The second, called circular topology, is the topology of the cyclic ordering, where, whenever $x<y<z$ in $\bfS$, we prescribe $x^-<x^+<y^-<y^+<z^-<z^+$. (Here, in a cyclic ordering, by $x_1<x_2<\dots<x_n$ we mean that $x_i<x_j<x_k$ for all $1\le i<j<k\le n$.) The circular topology is compact Hausdorff, totally disconnected, but not metrizable (since the set of clopen subsets is uncountable).

The interest is that the group $\PC^\bw(\bfS)$ naturally acts on $\bfS^\pm$, using one-sided limits in the obvious way, and this action preserves the circular topology. This makes the projection map $\bfS^\pm\to\bfS$, $y\mapsto\bar{y}$, near $\PC^\bw(\bfS)$-equivariant.

\subsection{Proof of realizability}

We need the following fact about the Denjoy blow-up.

\begin{prop}\label{cleanhat}
For every $g\in\PC^\bw(\bfS)$, the set of $x\in\bfS^\pm$ such that $g(x)=\hat{x}$ is finite.
\end{prop}
\begin{proof}
If by contradiction it is infinite, it has an accumulation point; conjugating by a suitable element of the isometry group $\mathrm{O}(2)$, we can suppose that this accumulation point is $0^+$. Hence, there is an injective sequence $(x_n)$ tending to $0^+$ such that $g(x_n)=\hat{x}_n$ for every $n$. There exists $\eps\in\mathopen]0,1\mathclose[$ such that $g$ induces a continuous (necessarily strictly monotone) function $\bar{g}$ on $\mathopen]0,\eps\mathclose[$, valued in $\mathopen]0,1\mathclose[$. Extracting, we can suppose that $0<\bar{x}_{n+1}<\bar{x}_n<\eps$ for all $n$.

On the one hand, since $g(x_n)=\hat{x}_n$ for every $n$, $\bar{g}$ is necessarily decreasing on $\mathopen]0,\eps\mathclose[$. On the other hand, since $g(x_n)=\hat{x}_n$, we have $\bar{g}(\bar{x}_n)=\bar{x}_n$ for all $n$, which implies that $\bar{g}$ is increasing on $\mathopen]0,\eps\mathclose[$. Contradiction.
\end{proof}

\begin{thm}
Every finitely generated abelian subgroup $A$ of $\PC^\bw(\bfS)$ is realizable (for its near action on $\bfS$).
\end{thm}
\begin{proof}
We use the Denjoy blow-up map $\pi:\bfS^\pm\to\bfS$. Here $\bfS^\pm$ is an $A$-set, the map $\pi$ has fibers $\{x,\hat{x}\}$ of cardinal 2 and is near equivariant as well as the involution $x\mapsto\hat{x}$; moreover it satisfies the additional assumption ($\ast$) of Theorem \ref{rearran}, by Proposition \ref{cleanhat}. Hence, Theorem \ref{rearran} applies and $\bfS$ is a realizable $A$-set.
\end{proof}

\begin{rem}
One step to the theorem was to show that $A$ is stably realizable. This step is much easier when $A\subset\IET^\bw$, or more generally when the near action of $A$ is piecewise analytic. Indeed, in this case, the criterion of Proposition \ref{complab} can be checked directly, as the set of fixed points of any finitely generated subgroup is then a Boolean combination of intervals.
\end{rem}

\section{Non-realizability of groups of interval exchanges with flips}\label{nriet}

\subsection{Non-realizability}

For a subgroup $\Lambda$ of $\R/\Z$, let $\IET^\bw_\Lambda$ be the subgroup of $\IET^\bw$ of elements with discontinuities in $\Lambda$, and local isometries of the form $x\mapsto\pm x+\lambda$ with $\lambda\in\Lambda$. Define $\IET^+_\Lambda=\IET^\bw_\Lambda\cap\IET^+$.

In $\IET^+$, we have the subgroup of genuine {\bf rotations} $r_t:x\mapsto x+t$ (translations of the group $\bfS=\R/\Z$). We endow $\mathbf{S}$ with its geodesic distance.

Given any interval $I$ in $\mathbf{S}$ (of measure in $\mathopen]0,1\mathclose[$), we have a corresponding subgroup of {\bf partial rotations}, acting trivially outside $I$, and acting as genuine rotations on $I$ when we ``close it" by identifying endpoints of $\bar{I}$. 

Given $f\in\PC^\bw(\mathbf{S})$, define its {\bf essential support} $\essup(f)\subset\bfS$ as the closure of the set of $x$ such that $(f(x^-),f(x^+))\neq (x^-,x^+)$. It is empty if and only if $f$ is the identity.
When $f\in\IET^\bw$ and $\tilde{f}$ is a lift, note that $\mathrm{essupp}(f)$ has finite symmetric difference with $\{x:\tilde{f}(x)\neq x\}$.

For $f\in\PC^\bw(\mathbf{S})$, define $\underline{\sing}(f)=\{x\in\bfS^\pm:f(\hat{x})\neq \widehat{f(x)}\}$. It is finite, and obviously invariant under $x\mapsto \hat{x}$; let $\sing(f)$ be its image in $\bfS$. This is the set of points at which every lift of $f$ is discontinuous.

For $f\in\IET^+(\mathbf{S})$, choose a representative $g(x)=x+\tau_g(x)$. Note that while $\tau_g$ depends on the choice of $g$, the values $\tau_g(x^-)$ and $\tau_g(x^+)$ (in $\mathbf{S}^\pm$) only depend on $f$. We write $\tau_f^-(x)=\overline{\tau_g(x^-)}$ and $\tau_f^+(x)=\overline{\tau_g(x^+)}$ (these elements of $\mathbf{S}$ are just the one-sided limits at $x$). Thus $\sing(f)=\{x:\tau_f^-(x)\neq \tau_f^+(x)\}$. Write $\nu_f(x)=\tau_f^+(x)-\tau_f^-(x)$, so $\sing(f)=\{x:\nu_f(x)\neq 0\}$. 

Let $E(f)\in\mathopen]0,1/2]\cup\{\infty\}$ be the minimal distance between any two points of the finite subset $\sing(f)$ (where $E(f)=\infty$ if $\sing(f)$ is empty, i.e., if $f$ is a rotation). 

When we consider the action of $\PC^\bw(\bfS)$ on subsets of $\bfS$, it is only well-defined modulo finite symmetric difference with finite subsets. We then talk of near subset, near disjoint (= finite intersection), near partition, etc.

\begin{lem}\label{comrot}
For every $f\in\IET^+$ and $t\in \mathopen]0,E(f)\mathclose[$, the ``commutator" $c=f^{-1}r_tfr_t^{-1}$ permutes the near intervals $[x,x+t]$, $x\in\sing(f)$, by translations, without preserving any of them. These intervals are pairwise near disjoint, and the essential support is $\bigsqcup_{x\in\sing(f)}[x,x+t]$.
\end{lem}
\begin{proof}
In this proof, ``generic" means ``with finitely many exceptions", and we freely choose representatives.

If $f$ is a rotation then $c$ is the identity. Otherwise, $f$ has at least two singularities. Let $a,b$ be consecutive singularities of $f$. Then the representative of $b-a$ in $\mathopen]0,1\mathclose[$ is $\ge t$. So we can view the interval $[a,b]$ as concatenations of intervals $[a,a+t]$ and $[a+t,b]$, and for generic $x\in [a+t,b]$, we have $f(x)=f(x-t)+t$. For a generic $x\in [a+t,b]$, we have $c(x)=f^{-1}r_tf(x-t)=f^{-1}r_t(f(x)-t)=f^{-1}(f(x))=x$. 

For  generic $x\in [a,a+t]$, we have $c(x)=f^{-1}r_tf(x-t)=f^{-1}r_t(x-t+s_f^-(a))=f^{-1}(x+s_f^-(a))$. Observe that $\overline{f(a^-)}=a+\tau_f^-(a)$ belongs to $\sing(f^{-1})$. Since $E(f)=E(f^{-1})$, this implies that for $x\in\mathopen]a,a+t\mathclose[$, $x+\tau_f^-(a)$ meets no singularity of $f^{-1}$.  Hence $c$ has no singularity in $\mathopen]a,a+t\mathclose[$. In addition, for generic $x\in [a,a+t]$, we have $f(x-t)\neq f(x)-t$, and hence $c(x)=f^{-1}(t+f(x-t))\neq f^{-1}(f(x))=x$.

Thus the image by $c$ of $[a,a+t]$ is (essentially) an interval of length $t$, included in the union of the intervals $[x,x+t]$, for $x\in\sing(f)$. Therefore, it is exactly one of these intervals (and not $[a,a+t]$).
\end{proof}

The diameter of a metric space is the supremum of distances between any two points.

\begin{lem}\label{arbsmall}
Let $\Gamma$ be a subgroup of $\IET^+(\mathbf{S})$. Suppose that $\Gamma$ includes a dense subgroup of rotations, and, for some proper sub-interval, some dense subgroup of the corresponding partial rotations. Then $\Gamma$ contains non-identity elements whose essential support has arbitrary small diameter.
\end{lem}
\begin{proof}
Up to conjugating by a rotation, we can suppose that $\Gamma$ includes a dense subgroup of partial rotations, namely of the interval $[0,\rho]$ for some $\rho\in\mathopen] 0,1\mathclose[$.

Fix $\eps_0>0$ small enough (see below), and let us produce non-identity element whose essential support has diameter $\le 5\eps_0$. We choose some $\eps\in\mathopen]0,\eps_0\mathclose[$ with the two additional conditions:
\begin{itemize}
\item $\eps<\min(\rho,1-\rho)/5$;
\item defining the partial rotation $f$ of $[0,\rho]$, of length $-2\eps$, the value of $\eps$ has been chosen so that $f\in\Gamma$.
\end{itemize}

Hence $s_f$ generically equals $\rho-2\eps$ on $[0,2\eps]$, $-2\eps$ on $[2\eps,\rho]$, and 0 on $[\rho,1]$. Choose $\eta$ with $0<\eta\le\eps$, such that $r_\eta\in\Gamma$. Then the essential support of $c=f^{-1}r_\eta fr_\eta^{-1}$ is, by Lemma \ref{comrot}, equal to $[0,\eta]\cup [2\eps,2\eps+\eta]\cup [\rho,\rho+\eta]$. Consider in $\Gamma$ a rotation of length $\lambda\in\mathopen]3\eps,4\eps\mathclose[$. Conjugate the given group of partial rotations by this rotation to obtain a dense group of partial rotations on $[\lambda,\lambda+\rho]$. (By the condition on $\eps$, we have $\lambda<\rho$ and $\lambda+\rho<1$.) Then there exists $q$ in this dense group of rotations of $[\lambda,\lambda+\rho]$ (essentially) mapping $[\rho,\rho+\eta]$ into $\mathopen]\lambda,5\eps[$, say $[\lambda',\lambda'+\eta]$ with $3\eps<\lambda'<4\eps$. Then the essential support of $q^{-1}cq$ is $[0,\eta]\cup [2\eps,2\eps+\eta]\cup [\lambda',\lambda'+\eta]$, and thus has diameter $\le 5\eps\le 5\eps_0$.
\end{proof}

\begin{defn}\label{def_clean}
We say that a subgroup of $\widehat{\PC^\bw}(\mathbf{S})$ is 
\begin{itemize}
\item clean if its intersection with the group of finitely supported permutations is trivial;
\item hyper-clean if for every $g$ in the subgroup, the graph of $g$ (viewed as subset of $\bfS\times\bfS$) has no isolated point; equivalently if, at every point, $g$ is either left or right-continuous.
\end{itemize}
\end{defn}

\begin{lem}\label{thenloccl}
Let $\tilde{\Gamma}$ be a clean subgroup of $\widehat{\PC^\bw}(\mathbf{S})$, and $\Gamma$ its image in $\PC^\bw(\mathbf{S})$. Suppose that $\tilde{\Gamma}$ includes a dense subgroup $\tilde{\Lambda}$ of rotations. Suppose that $\Gamma$ admits non-identity elements with essential support of arbitrary small diameter. Then $\tilde{\Gamma}$ is hyper-clean. 
\end{lem}
\begin{proof}
For $h\in\widehat{\PC^\bw(\mathbf{S})}$, define its interior support as the intersection of the support $\{x:hx\neq x\}$ with the set of continuity points of $h$. It is open, and is included and cofinite in the $h$-invariant subset $\{x:hx\neq x\}$, and it is also included and dense in the essential support of its image $\bar{h}$ in $\PC^\bw(\mathbf{S})$.

In a first step, we show that $\tilde{\Gamma}$ is locally clean, in the sense that for each element $g$, the support $\{x:gx\neq x\}$ has no isolated point.

By contradiction, let $x$ be an isolated non-fixed point of $g$. For some $\eps>0$, all other points in $\mathopen] x-\eps,x+\eps\mathclose[$ are fixed by $g$. There exists $h_0\in\tilde{\Gamma}$ whose essential support has diameter $<\eps$. Hence there exists some conjugate $h$ of $h_0$ by some element of $\tilde{\Lambda}$ such that both $x$ and $h(x)$ belong to the interior support of $h$ (indeed, letting $I$ be the interior support of $h_0$ and $J=h_0^{-1}I\cap I$, which is cofinite in $I$, it is enough to find $s\in\tilde{\Lambda}$ such that $x\in sI$ and define $h=sh_0s^{-1}$). Hence the essential support of $h$ is included in $\mathopen] x-\eps,x+\eps\mathclose[$; in particular, $\bar{h}$ and $\bar{g}$ commute. Since $\tilde{\Gamma}$ is clean, it follows that $h$ and $g$ commute. Since $g(x)\neq x$, it follows that $g(h(x))\neq h(x)$. But $h(x)\neq x$, and $h(x)$ belongs to the interior support, which is included in $\mathopen] x-\eps,x+\eps\mathclose[$, hence $h(x)$ is fixed by $g$. This is a contradiction, concluding the first step.

Now let us prove that $\tilde{\Gamma}$ is hyper-clean. Suppose by contradiction that $\{x,g(x)\}$ is an isolated point in the graph of $g$. Up to post-compose $g$ with a nontrivial rotation, we can suppose that $g(x)\neq x$. So there exists $\eps>0$ such that none of $x$, $g(x^+)$, $g(x^-)$ belongs to $\mathopen]g(x)-2\eps,g(x)+2\eps\mathclose[$.

As in the proof of the first step (using the dense subgroup of rotations), let $h\in\tilde{\Gamma}$ have essential support of diameter $<\eps$, with both $g(x)$ and $h(g(x))$ in the interior support of $h$; we can also require that $h(x)=x$.

Then for $t\in\mathopen]x-\eps,x+\eps\mathclose[$, $t\neq x$, we have $t\notin\mathopen]g(x)-\eps,g(x)+\eps\mathclose[$ and $g(t)\notin\mathopen]g(x)-\eps,g(x)+\eps\mathclose[$. Hence, with finitely many exceptions on $t$, we have $g(h(t))=g(t)$ and $h(g(t))=g(t)$. Also, we have $g(h(x))=g(x)$ and $h(g(x))\neq g(x)$. Hence, $g^{-1}h^{-1}gh$ has an isolated non-fixed point at $x$. This contradicts the assumption that $\tilde{\Gamma}$ is locally clean.
\end{proof}

\begin{defn}\label{132flip}
Call an element $f$ of $\PC^\bw(\mathbf{S})$ a 132-flip if it satisfies: there are three nonzero consecutive intervals $I_1$, $I_2$, $I_3$ near partitioning $\mathbf{S}$, such that
\begin{enumerate}
\item $f$ has no singularity in the interior of $I_j$ for each $j=1,2,3$;
\item $f(I_1)=I_2$, $f(I_2)=I_3$, $f(I_3)=I_3$;.
\item $f:I_j\to f(I_j)$ is orientation-reversing for $j=1$ and orientation-preserving for $j=2,3$;
\item $f^2$ is the identity.
\end{enumerate}

Say that $f$ is a triple flip if $f^2$ is the identity, and there are three nonzero consecutive intervals $I_1$, $I_2$, $I_3$ near partitioning $\mathbf{S}$, such that $f$ essentially preserves $I_j$ for each $j$, and is orientation-reversing on it.
\end{defn}

\begin{lem}\label{nonliftnew}
Let $f\in\PC^\bw(\mathbf{S})$ be a 132-flip, or a triple flip. Then $f$ has no hyper-clean lift squaring to the identity.
\end{lem}
\begin{proof}
\begin{figure}[h]
  \begin{minipage}[c]{.4\textwidth}
    \includegraphics[width=\textwidth]{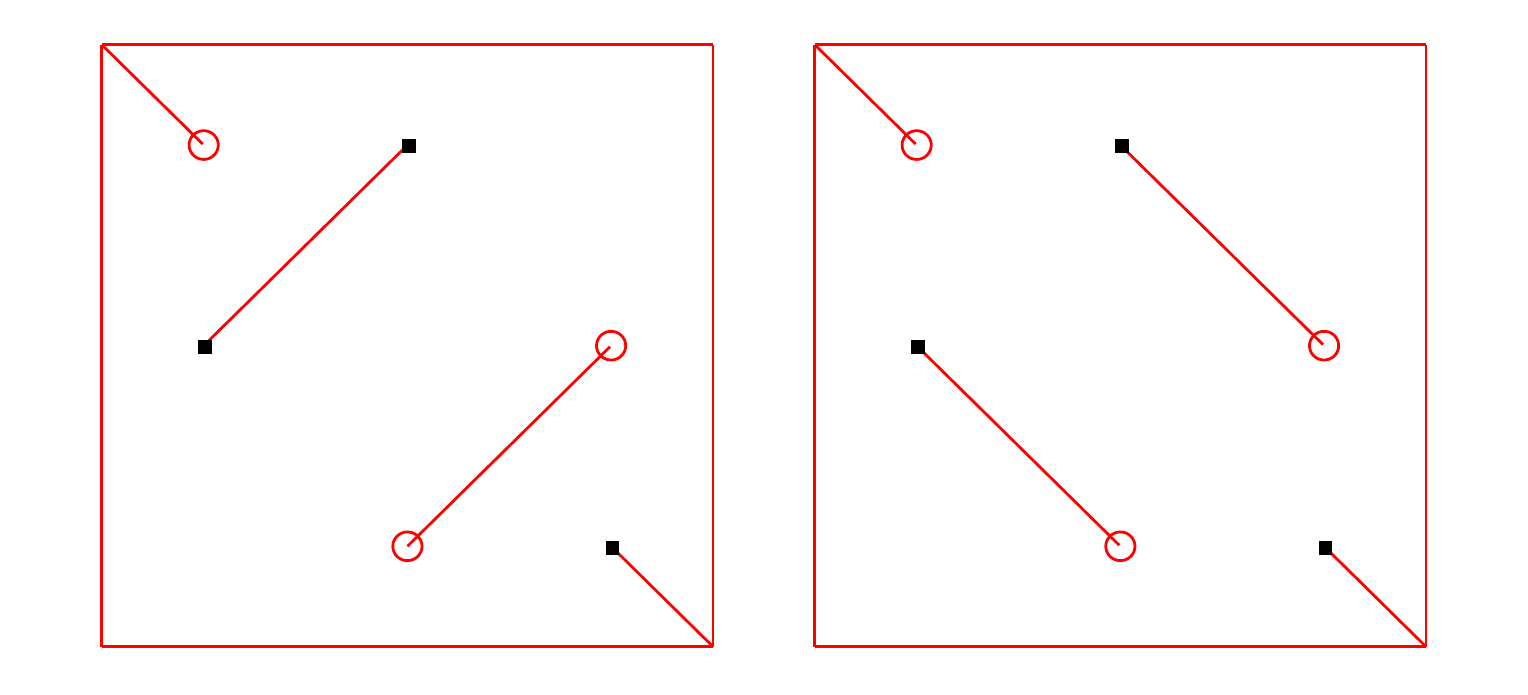}
  \end{minipage}\hfill
  \begin{minipage}[c]{.6\textwidth}
    \caption{Graphs of a 132-flip and a triple flip: there are two hyper-clean lifts, choosing either the endpoints denoted as circles or dots.
    } \label{fig132bis}
  \end{minipage}
\end{figure}
For a 132-flip, conjugating by a rotation, we can suppose that for some $0<a<b<1$, $f$ preserves and is orientation-reversing on $[0,a]$, exchanges $[a,b]$ and $[b,1]$ in an orientation-preserving way. If $q$ is a hyper-clean lift of $f$, then, viewing $q$ as subset of $\mathbf{S}\times\mathbf{S}$, we have $(0,a)\in q$ $\Leftrightarrow$ $(b,a)\notin q$ $\Leftrightarrow$ $(b,0)\in q$ $\Leftrightarrow$ $(a,0)\notin q$ $\Leftrightarrow$ $(a,b)\in q$ $\Leftrightarrow$ $(0,b)\notin q$. Hence $q$ acts on $\{0,a,b\}$ as a 3-cycle, and in particular $q^2$ is not the identity.

The case of a triple flip is similar; again, we obtain that there are exactly two hyper-clean lifts, both of order 6. This is illustrated in Figure \ref{fig132bis}.
\end{proof}

\begin{thm}\label{nonrepm}
Let $\Gamma$ be a subgroup of $\PC^\bw(\mathbf{S})$. Suppose that
\begin{enumerate}
\item $\Gamma$ includes a subgroup $\Lambda$ of rotations of $\Q$-rank $\ge 2$, or infinitely generated and of $\Q$-rank 1;
\item $\Gamma$ includes, for some proper nonzero interval, some dense subgroup of the corresponding partial rotations;
\item $\Gamma$ contains an 132-flip or a triple flip (Definition \ref{132flip}).
\end{enumerate}
Then $\Gamma$ is not realizable.
\end{thm}
\begin{proof}
By contradiction, let $\tilde{\Gamma}$ be a lift; then $\tilde{\Gamma}$ is clean. The assumption on $\Lambda$ implies that it is 1-ended and not locally finite. By Proposition \ref{per1e}(\ref{pe1}), we can, after conjugation, assume that $\Lambda$ lifts as a subgroup of genuine rotations. By Lemma \ref{arbsmall}, $\Gamma\cap\IET^+$ contains non-identity elements of arbitrary small essential diameter. By Lemma \ref{thenloccl}, $\tilde{\Gamma}$ is hyper-clean. Finally Lemma \ref{nonliftnew} yields a contradiction.
\end{proof}

\begin{cor}
$\IET^\bw$ is not realizable (for its near action on $\mathbf{S}$), as well as its subgroup $\IET^\pm$. Moreover, the latter admits a non-realizable finitely generated subgroup.
\end{cor}
\begin{proof}
To fulfill the assumptions of Theorem \ref{nonrepm}, consider a pair $(u,v)$ of rotations generating a subgroup of $\R/\Z$ of $\Q$-rank 2; a partial rotation $w$ generating a dense subgroup in a proper interval, and a triple flip $z\in\IET^\pm$ as in Lemma \ref{nonliftnew}. Then $\{u,v,w,z\}$ generates a non-realizable subgroup, by Theorem \ref{nonrepm}.
\end{proof}

\begin{cor}
Let $\Lambda$ be a subgroup of $\R/\Z$, of $\Q$-rank $\ge 2$, or infinitely generated of $\Q$-rank $1$. Then $\IET^\bw_\Lambda$ is not realizable, and neither is its subgroup $\IET^\pm_\Lambda=\IET^\bw\Lambda\cap\IET^\pm$.\qed
\end{cor}

\subsection{Restricted realizability}

Recall that a partial rotation is an element of $\IET^+$ which, for a partition of the circle into three consecutive (possibly empty) intervals, exchanges two of them and (pointwise) fixes the third one.

\begin{lem}\label{partrot}
Let $r$ be a hyper-clean lift of a partial rotation of the circle. Then $r$ is either left or right-continuous.
\end{lem}
\begin{proof}
We can conjugate by a rotation to suppose, for convenience, that the partial rotation has support $[a,c]$ for $0\le a<c<1$, and, for some $b,b'\in\mathopen] a,c\mathclose[$, (essentially) maps $[a,b]$ onto $[b',c]$ and $[b,c]$ onto $[a,b']$. Then, viewing $r$ as a subset of $\mathbf{S}\times\mathbf{S}$, we have $(a,a)\in r$ $\Leftrightarrow$ $(a,b')\notin r$ $\Leftrightarrow$ $(c,b')\in r$
$\Leftrightarrow$ $(c,c)\notin r$ $\Leftrightarrow$ $(b,c)\in r$ $\Leftrightarrow$ $(b,a)\notin r$. Hence, $r$ either contains the three elements $(a,a)$, $(c,b')$, $(b,c)$ and is left-continuous, or contains the other three pairs and is right-continuous.
\end{proof}

\begin{lem}\label{allpartrot}
Let $\tilde{\Gamma}$ be a hyper-clean subgroup of $\widehat{\IET^+}$ and $\Gamma$ its image in $\IET^+$. Suppose ($\#$) that the group of rotations in $\Gamma$ achieves all translation lengths of $\Gamma$.
 Suppose that each translation length of every element of $\Gamma$ is achieved by a rotation belonging to $\tilde{\Gamma}$. Then either all partial rotations in $\tilde{\Gamma}$ are left-continuous, or all are right-continuous.
\end{lem}
\begin{proof}
By ($\#$), whenever we have a partial rotation in $\Gamma$, the rotation moving by the length of its support interval belongs to $\Gamma$. Therefore, supposing that the conclusion fails, after conjugation and using Lemma \ref{partrot}, we can suppose that we have two partial rotations $r,s$ with left endpoint $0$, with $r$ right but not left-continuous, and $s$ left but not right-continuous. 

Then $s(0)=s(0^-)=0$ and $s(0^+)\neq 0$. So $r(s(0))=r(0)\neq 0$, while $r(s(0^-))=0$. Hence $rs$ is not left-continuous at 0. Since $\tilde{\Gamma}$ is hyper-clean, this implies that $rs$ is right-continuous at 0. Hence $rs(0^+)=rs(0)$. So $rs(0^+)=r(0)=r(0^+)$. Hence $s(0^+)=0$; this is a contradiction.  
\end{proof}

For $\Gamma\le\PC^+(\mathbf{S})$, denote by $\Gamma_{\mathrm{left}}$ (respectively $\Gamma_{\mathrm{right}}$) the group of left-continuous (resp.\ right-continuous) representatives of elements of $\Gamma$.

\begin{thm}\label{t2lifts}
Let $\tilde{\Gamma}$ be a clean subgroup of $\widehat{\IET^+}$ and $\Gamma$ its image in $\IET^+$. Suppose that
\begin{enumerate}
\item the group of rotations $\Lambda$ in $\Gamma$ achieves all translation lengths of $\Gamma$;
\item $\Lambda$ and has $\Q$-rank $\ge 2$, or is infinitely generated of $\Q$-rank $\ge 1$;
\item $\Gamma$ is generated by its partial rotations;
\item $\Gamma$ includes a dense subgroup of partial rotations (for some proper sub-interval).
\end{enumerate}
 Then $\tilde{\Gamma}$ is conjugate, by a (unique) finitely supported permutation, to either $\Gamma_{\mathrm{left}}$ or $\Gamma_{\mathrm{right}}$.
\end{thm}
\begin{proof}
Using the $\Q$-rank assumption, by Proposition \ref{per1e}(\ref{pe1}), we can conjugate by a finitely supported permutation, to ensure that rotations indeed act by rotations. In this case, we will prove that $\tilde{\Gamma}$ is then {\it equal} to either $\Gamma_{\mathrm{left}}$ or $\Gamma_{\mathrm{right}}$.

By the assumption of existence of both a dense subgroup of rotations and a dense partial subgroup of partial rotations, Lemma \ref{arbsmall} implies that $\Gamma$ admits elements of arbitrary small essential diameter. In turn, again using a dense subgroup of rotations, and the existence of elements of arbitrary small essential diameter, Lemma \ref{thenloccl} ensures that $\tilde{\Gamma}$ is hyper-clean. Since $\tilde{\Gamma}$ is hyper-clean and all its translations length are achieved by rotations, we apply Lemma \ref{allpartrot} to ensure that all partial rotations are, say, left-continuous (the right-continuous case is equivalent up to conjugate by a reflection). We conclude, since $\Gamma$ is generated by its partial rotations.
\end{proof}

\begin{cor} Suppose that $\Lambda$ has $\Q$-rank $\ge 2$, or is infinitely generated of $\Q$-rank 1. Then $\IET^+_\Lambda$ has only two lifts to $\widehat{\IET^+}$ up to conjugation by finitely supported permutations, namely the left-continuous and the right-continuous lift.
\end{cor}
\begin{proof}
We have to check that the assumptions of Theorem \ref{t2lifts} are fulfilled. The group of translations lengths in $\IET^+_\Lambda$ and the group of rotations of $\IET^+_\Lambda$ are both equal to $\Lambda$. Also, for every $x<x_0\in\mathopen] 0,1\mathclose[$ representing elements of $\Lambda$, the corresponding partial rotation (exchanging $[0,x\mathopen[$ and $[x,x_0\mathclose[$ belongs to $\IET^+_\Lambda$; hence they achieve ($x_0$ fixed, $x$ varying) a dense subgroup of partial rotations (here we only use that $\Lambda$ is dense). Finally, that $\IET^+_\Lambda$ is generated by its partial rotations is proved in Lemma \ref{genparro}.
\end{proof}

\begin{lem}\label{genparro}
For every subgroup $\Lambda$ of $\R/\Z$, the group $\IET^+_\Lambda$ is generated by its partial rotations.
\end{lem}
\begin{proof}
Every element of $\IET^+$ can be represented as $T=[\![u,\sigma]\!]$ where, for some $n$, $u\in\{(u_1,\dots,u_n)\in\R_+^n:\sum u_i=1\}$ and $\sigma\in\mks_n$ (where $\R_+$ denotes non-negative reals). Recall that, denoting $X_i=[\sum_{j=1}^{i-1}u_j,\sum_{j=1}^iu_j\mathclose[$, the right-continuous representative $\tilde{T}$ of the transformation $T$ consists in rearranging the intervals $X_1,\dots,X_n$, moving $X_i$ in position $\sigma(i)$. The precise formula \cite{Kea} is given by 
\[\tilde{T}x=x-\sum_{j=1}^{i-1}u_j+\sum_{j=1}^{\sigma(i)-1}u_{\sigma^{-1}(j)},\quad x\in X_i.\]
In particular, the singularities belong to $\{0,u_1,u_1+u_2,\dots,u_1+\dots u_{n-1}\}$, which is included in the subgroup generated by the $u_i$. 
Say that $\sigma$ is admissible if $\sigma(i+1)\neq\sigma(i)+1$ for all $i<n$: this means that $n$ is minimal among possible representatives $[\![u,\sigma]\!]$. Assuming that $\sigma$ is admissible, all these elements are singularities (in the interval model, i.e., we always consider 0 and $\tilde{T}^{-1}(0)$ as singularities), so the subgroup generated by singularities equals the subgroup $\Xi_T$ generated by the $u_i$, and hence achieves all translation lengths.

The composition formula is given by
\[[\![\psi\circ\sigma,u]\!]=[\![\psi,\sigma\cdot u]\!]\circ [\![\sigma,u]\!],\quad \text{where }(\sigma\cdot u)_i=u_{\sigma^{-1}(i)},\]
which, iterating (and omitting signs $\circ$), yields
\[[\![\sigma_m\dots\sigma_1,u]\!]=[\![\sigma_m,\sigma_{m-1}\dots\sigma_1u]\!]\dots [\![\sigma_2,\sigma_1u]\!][\![\sigma_1,u]\!].\]

Note that each $[\![\sigma_i,\sigma_{i-1}\dots\sigma_1u]\!]$ has all its singularities and translation lengths in $\Xi_u$.

Now consider $T\in\IET^+_\Lambda$. Write $T=[\![\sigma,u]\!]$ with $\sigma$ admissible. Then $\Xi_u\subset\Lambda$. Write $\sigma=\psi_m\dots\psi_1$ where each $\psi_i$ is a transposition $(n_i,n_i+1)$. Since $\Xi_u\subset\Lambda$, we deduce that $\psi_i\in\IET^+_\Lambda$. Since $\psi_i$ is a transposition of two consecutive elements, for every $v$, $[\![\psi_i,v]\!]$ is a partial rotation. We deduce that $\IET^+_\Lambda$ is generated by its partial rotations.
\end{proof}

\begin{cor}\label{2lifts_cor}
Let $\Gamma$ be a subgroup of $\PC^+(\mathbf{S})$ including $\IET^+$ (e.g., $\PC^+(\mathbf{S})$, or its piecewise analytic subgroup). Then $\Gamma$ has, up to conjugation by a (unique) finitely supported permutation, only the two lifts $\Gamma_{\mathrm{left}}$ and $\Gamma_{\mathrm{right}}$. 
\end{cor}
\begin{proof}
The uniqueness is clear. Let $\tilde{\Gamma}$ be a lift. By Theorem \ref{t2lifts}, we can suppose (up to conjugate by a finitely supported permutation and possibly by a reflection, that in restriction to $\IET^+$, we have the left-continuous lift. By Lemma \ref{thenloccl}, $\tilde{\Gamma}$ is hyper-clean. 

Let $g$ be an element of $\tilde{\Gamma}$ and $x\in\mathbf{S}$. Then there exists an interval exchange $h$ such that $h(x^-)=g(x^-)$ and $h(x^+)=g(x^+)$. We can view $h$ as an element of $\tilde{\Gamma}$, and thus $h$ is left-continuous, so $h(x)=h(x^-)$. We have $h^{-1}g(x^-)=x^-$ and $h^{-1}g(x^+)=x^+$. Since $\tilde{\Gamma}$ is hyper-clean, we deduce that $h^{-1}g(x)=x$. So $g(x)=h(x)=h(x^-)=g(x^-)$, showing that $g$ is left-continuous at $x$.
\end{proof}

\subsection{Stable non-realizability}

Let $X$ be a set. Let $\widehat{\PC^\bw}(\mathbf{S}\sqcup X,\bfS)$ be the subgroup of permutations of $\mathbf{S}\sqcup X$ that are identity on a cofinite subset of $X$, and that induce elements of $\PC^\bw(\mathbf{S})$ on $\mathbf{S}$. So there is a canonical projection $\widehat{\PC^\bw}(\mathbf{S}\sqcup X,\bfS)\to \PC^\bw(\mathbf{S})$, whose kernel $\mathfrak{S}_{\mathrm{fin}}(\mathbf{S}\sqcup X)$ consists of finitely supported permutations of $\mathbf{S}\sqcup X$. We say that a subgroup $\Gamma$ of $\widehat{\PC^\bw}(\mathbf{S}\sqcup X,\bfS)$ is {\bf clean} if it has trivial intersection with $\mathfrak{S}_{\mathrm{fin}}(\mathbf{S}\sqcup X)$.

For $f\in\widehat{\PC^\bw}(\mathbf{S}\sqcup X,\bfS)$, we call {\bf essential support} of $f$ the closure of the set of $x\in\mathbf{S}$ such that $f(x)\notin\big\{\overline{f(x^+)},\overline{f(x^-)}\big\}$. (Recall that $y\mapsto\bar{y}$ denotes the projection $\mathbf{S}^\pm\to\mathbf{S}$.)

Here is an adaptation of Lemma \ref{thenloccl}.

\begin{lem}\label{thenloccl2}
Let $\tilde{\Gamma}$ be a clean subgroup of $\widehat{\PC^\bw}(\mathbf{S}\sqcup X,\bfS)$, and $\Gamma$ its image in $\PC^\bw(\mathbf{S})$. Suppose that $\tilde{\Gamma}$ includes a dense subgroup $\tilde{\Lambda}$ of rotations acting as the identity on $X$. Suppose that $\Gamma$ admits non-identity elements with essential support of arbitrary small diameter. Then
\begin{enumerate}
\item \label{hypcle2} for every $g\in\tilde{\Gamma}$ and $x\in\mathbf{S}$, $g(x)$ belongs to $\big\{\overline{g(x^+)},\overline{g(x^-)}\big\}\cup X$.
\item\label{loccle2} for every $g\in\tilde{\Gamma}$, the subset $\{x\in\mathbf{S}:gx\neq x\}$ has no isolated point.
\end{enumerate}
\end{lem}
\begin{proof}
(\ref{hypcle2}) This is an adaptation of the proof of Lemma \ref{thenloccl} and we skip details; note that when $X$ is empty this is precisely the same statement. The first step consists in proving that for $x\in\mathbf{S}$ if $\overline{g(x^-)}=\overline{g(x^+)}=x$, then $g(x)\in\{x\}\cup X$. The second step assumes that $g(x)\in\mathbf{S}$ and $g(x)\notin\{g(x^-),g(x^+)\}$ and reaches a contradiction.

(\ref{loccle2}) Suppose by contradiction that there exists $x\in\mathbf{S}$ such that $f(x')=x'$ for all $x'\neq x$ close enough to $x$, but $f(x)\neq x$. By (\ref{hypcle2}), we have $f(x)\in X$. Let $r\in\tilde{\Lambda}$ be a small enough non-trivial rotation and $x'\neq x$ close enough to $x$ (``close enough" may depend on $r$). Then $r^{-1}f^{-1}rf(x')=x'$, while $r^{-1}f^{-1}rf(x)=r^{-1}f^{-1}f(x)=r^{-1}x\neq x$. Since $r^{-1}x\in\mathbf{S}$, this contradicts (\ref{hypcle2}).
\end{proof}

Say that an element $g$ of $\PC^\bw(\mathbf{S})$ has small support if there exists a rotation $r$ such that $rS\cap S=\emptyset$, where $S$ is the essential support of $g$.

\begin{lem}\label{Xinv}
Let $\tilde{\Gamma}$ be a clean subgroup of $\widehat{\PC^\bw}(\mathbf{S}\sqcup X,\bfS)$ such that for every $g\in\tilde{\Gamma}$, the subset $\{x\in\mathbf{S}:gx\neq x\}$ has no isolated point. Let $\Gamma$ be its projection to $\PC^\bw(\mathbf{S})$.

Suppose that $\tilde{\Gamma}$ includes a dense subgroup $\tilde{\Lambda}$ of rotations (acting as the identity on $X$). Let $f\in\Gamma$ be an element with small support, and $\tilde{f}$ its lift in $\tilde{\Gamma}$. Then $X$ is $\tilde{f}$-invariant.
\end{lem}
\begin{proof}
Let $T\subset\mathbf{S}$ be the essential support of $f$, and let $T'=\{x\in\mathbf{S}\sqcup X:f(x)\neq x\}$ be the support of $\tilde{f}$ (so $T\cap X$ is finite). By the assumption on isolated points, we have $T'\cap\mathbf{S}\subset T$.

 There exists a non-empty open interval of rotations mapping $T$ to a disjoint subset; fix one, say $\tilde{r}$, in $\tilde{\Lambda}$. Let $r$ denote its image in $\Gamma$.
 Hence $\tilde{r}(T')\cap T'$ is equal to $T'\cap X$, which is finite.
    
 Note that $\tilde{r}(T')$ is the support of $\tilde{r}\tilde{f}\tilde{r}^{-1}$. Since $r$ and $rfr^{-1}$ have essentially disjoint support, they commute, and hence $\tilde{f}$ and $\tilde{r}\tilde{f}\tilde{r}^{-1}$ commute. Thus $T'\cap X$ is $\tilde{f}$-invariant, and hence $X$ is $\tilde{f}$-invariant.
\end{proof}

\begin{thm}\label{t_stabr}
The near action of $\IET^\pm$ (and hence of $\IET^\bw$) on $\bfS$ is not stably realizable. Moreover,
\begin{enumerate}
\item there exists a finitely generated subgroup of $\IET^\pm$ whose near action on $\bfS$ is not stably realizable;
\item for every subgroup $\Lambda$ of $\R/\Z$ that has $\Q$-rank $\ge 2$, or infinitely generated of $\Q$-rank $1$, the near action of $\IET_\Lambda^\pm$ (and hence of $\IET_\Lambda^\bw$) on $\bfS$ is not stably realizable.
\end{enumerate}
\end{thm}
\begin{proof}
For two essentially disjoint intervals $I,J$ of the same nonzero size, let $\varphi(I,J)$ be the element of order 2 in $\IET^+$ exchanging $I$ and $J$ by an orientation-preserving isometry.

Consider $a\in\left]\frac18,\frac14\right[$ and $b=1-4a$, so $0<b<\frac12$ and $4a+b=1$. Define $u,v\in\IET^+$ by $u=\varphi([0,a],[2a+b,3a+b])$ and $v=\varphi([a,2a],[3a+b,1])$. Let $w$ be a partial rotation on the interval $[0,2a]$, of infinite order. Let $s$ be the orientation-reversing isometry flipping all $[0,1]$. 

\begin{figure}[h]
\includegraphics[width=13cm]{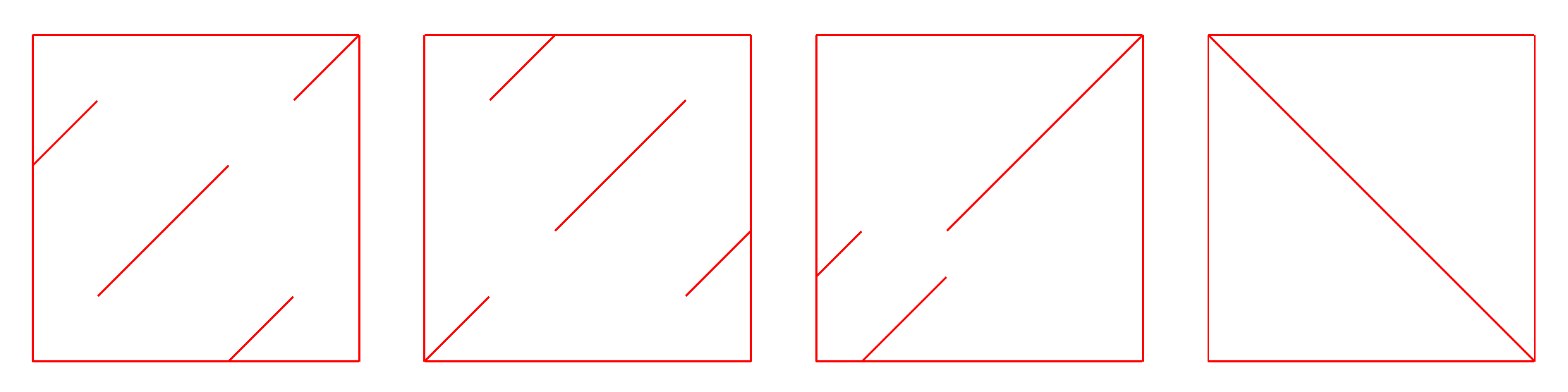}
\caption{Graphs of $u$, $v$, $w$ and $s$.}\label{fig9}
\end{figure}

Let $\Gamma$ be a subgroup of $\IET^\bw$ including a dense subgroup $\Lambda$ of rotations, either of $\Q$-rank $\ge 2$, or infinitely generated of $\Q$-rank~$1$. Furthermore we assume that $u$, $v$, $w$ and $s$ all belong to $\Gamma$. Let us show that the near action of $\Gamma$ on $\mathbf{S}$ is not stably realizable.

Arguing by contradiction, we assume it is realizable on $\mathbf{S}\sqcup X$ for some set $X$, and denote by $\tilde{\Gamma}$ a lift. Let $\Lambda$ be the group of rotations in $\Gamma$. After conjugation by a finitely supported permutation, we can assume, by Proposition \ref{per1e}(\ref{pe1}), that the lift $\tilde{\Lambda}$ acts by rotations on $\mathbf{S}$ (so it acts on $X$ by finitely supported permutations). Fix an element $r$ of infinite order in $\Lambda$. Since $srs^{-1}=r^{-1}$, we also have $\tilde{s}\tilde{r}\tilde{s}^{-1}=\tilde{r}^{-1}$. Hence $\tilde{s}$ normalizes $\langle\tilde{r}\rangle$, and thus $\tilde{s}$ preserves the union of all infinite $\langle\tilde{r}\rangle$-orbits, which is exactly $\mathbf{S}$. 

By Lemma \ref{Xinv} (which applies using Lemma \ref{thenloccl2}(\ref{loccle2})), the elements $\tilde{u}$, $\tilde{v}$, $\tilde{w}$ also preserve $\mathbf{S}$. Let $\Gamma_1$ be generated by $\{u,v,w,s\}\cup\Lambda$. Then $\tilde{\Gamma}_1$ preserves $\mathbf{S}$, so $X$ plays no longer any role: indeed since $\Gamma$ is clean, the action of $\tilde{\Gamma}_1$ on $\mathbf{S}$ is faithful. Since $w$ is a partial rotation (on $[1-2c,1]$) of infinite order, Lemma \ref{arbsmall} ensures that $\Gamma$ has non-identity elements with essential support of arbitrary small diameter. Then Theorem \ref{thenloccl} implies that $\tilde{\Gamma}_1$, viewed as subgroup of $\widehat{\PC^\pm}(\mathbf{S})$, is hyper-clean. 
The product $suv\in\Gamma_1$ is a triple flip on $\mathbf{S}$, but has no hyper-clean lift squaring to the identity (Lemma \ref{nonliftnew}). We reach a contradiction. 

Since $\Lambda$ can be chosen 2-generated, we obtain by construction a 6-generator subgroup that is not stably realizable. Moreover, if we restrict to a given $\Lambda$ as in the theorem, $a$ and then $w$ can be chosen to belong to $\IET_\Lambda^\pm$.
\end{proof}

\end{document}